\documentclass[12pt]{amsart}

\usepackage{amsthm}
\usepackage{amsmath}
\usepackage{amsfonts,amssymb,latexsym}
\usepackage[english]{babel}
\usepackage{epsfig,epsf}
\usepackage{tikz}
\usepackage{enumerate}
\usepackage{color}
\newtheorem{thm}{Theorem}
\newtheorem{lem}{Lemma}
\newtheorem{cor}{Corollary}

\newtheorem{prop}{Proposition}

\newtheorem{rem}{Remark}

\newcommand\subsetsim{\mathrel{%
  \ooalign{\raise0.2ex\hbox{$\subset$}\cr\hidewidth\raise-0.9ex\hbox{\scalebox{0.9}{$\sim$}}\hidewidth\cr}}}

\newcommand{\N}{{\mathbb N}}
\newcommand{\Z}{{\mathbb Z}}

\newcommand{\R}{{\mathbb R}}

\title[Additively stable sets, critical sets.]{Additively stable sets, critical sets for the $3K-4$ theorem in $\Z$ and $\R$}
\author{Paul P\'eringuey}

\address{Department of Mathematics, University of British Columbia, Vancouver, British Columbia, Canada\\ Pacific Institute for the Mathematical Sciences, Vancouver, British Columbia, Canada.}
\email{peringuey@math.ubc.ca}

\author{Anne de Roton}

\address{Universit\' e de Lorraine, Institut Elie Cartan de Lorraine, CNRS, Nancy, F-54000, France.}
\email{anne.de-roton@univ-lorraine.fr}

\begin{document}

\begin{abstract}
We describe in this paper additively left stable sets, i.e. sets satisfying $\left((A+A)-\inf(A)\right)\cap[\inf(A),\sup(A)]=A$ (meaning that $A-\inf(A)$ is stable by addition with itself on its convex hull), when $A$ is a finite subset of integers and when $A$ is a bounded subset of real numbers. More precisely we give a sharp upper bound for the density of $A$ in $[\inf(A),x]$ for $x\le\sup(A)$, and construct sets reaching this density for any given $x$ in this range. This gives some information on sets involved in the structural description of some critical sets in Freiman's $3k-4$ theorem in both cases.
\textbf{Mathematic Subject Classification:} 11P70.
\end{abstract}

\maketitle

\section{Introduction}

Given two nonempty subsets $A$ and $B$ of an additive group $G$, $A+B$ is the Minkowski sum of $A$ and $B$ defined by
$$A+B=\{a+b,\,{ with }\,a\in A,\, b\in B\}.$$
We will study the size of the sumset $A+B$ in two cases: when the ambient group $G$ is $\Z$ in which case the size of a finite set $A\subset \Z$ will be its cardinality and denoted by $|A|$ and when $G=\R$ in which case the size of a bounded set $A$ will be its inner Lebesgue measure $\lambda(A)$. \\
\\
In additive combinatorics, studying the structure of $A$, $B$ and $A+B$ when $A+B$ is small is an important topic. 
Using that $\min(A)+B$ and $\max(B)+A$ are both included in $A+B$ if $A$ and $B$ are finite sets of integers or closed and bounded sets of real numbers, we easily get that $|A+B|\ge |A|+|B|-1$ when we deal with subsets of $\Z$ and $\lambda(A+B)\ge \lambda(A)+\lambda(B)$ when $A$ and $B$ are subsets of $\R$. The sum $A+B$ being small means that its size is close to this lower bound.  \\
\\
When the ambient group is the set of integers $\Z$, Freiman's theorem (see for example \cite{Tao_Vu}) gives a description of finite sets $A$ such that $|A+A|\le C|A|$ where $C$ is a fixed constant. It asserts that $A$ has to be "efficiently" included in a generalized arithmetic progression of dimension bounded by a function depending on $C$. Many works are devoted to the search of the better bounds for both the dimension and the size of this generalized arithmetic progression in terms of $C$ (see for example \cite{schoen}). \\
When $C$ is very small (smaller than $3$), the so called Freiman $3k-4$ theorem, proved by Freiman in the symmetric case $A=B$ \cite{Freiman1959}, \cite{Freiman2009} and generalized to the sum of two distinct sets by Freiman  \cite{Freiman1962}, Lev and Smeliansky \cite{Lev_Smeliansky}, Stanchescu \cite{stanchescu} and Bardaji and Grynkiewicz \cite{Bardaji_Grynkiewicz}  gives a sharp statement describing the structures of the sets $A$, $B$ and $A+B$.
We state below a simplified version of this theorem.\\
\begin{thm}[Freiman $3k-4$]\label{3k-4}
Let $A$ and $B$ be finite nonempty subsets of $\Z$.\\ Assume $|B|\le |A|$ and $|A+B|\le |A|+2|B|-3-\delta$ with
$\delta=\begin{cases} 1& \mbox{ if $A$ is a translate of $B$},\\ 0 &\mbox{ otherwise}.\end{cases} $\\
Then there exist arithmetic progressions $P_A$, $P_B$ and $P_{A+B}$ of common difference 
such that
\begin{itemize}
\item $A\subset P_A$, $B\subset P_B$, $P_{A+B}\subset A+B$ and 
\item $|P_A|\le |A+B|-|B|+1$, $|P_B|\le |A+B|-|A|+1$, $|P_{A+B}|\ge |A|+|B|-1$.  
\end{itemize}
\end{thm}
Theorem \ref{3k-4} implies that any nonempty finite subset $A$ of $\Z$ satisfying $|A+A|\le 3|A|-4$ (hence the denomination $3k-4$ theorem, if $k$ is the size of the set) is included in a genuine arithmetic progression $P_A$ of size at most $|A+A|-|A|+1$. 
Freiman gave in \cite{Freiman09} a structural description of critical sets $A$ such that the smallest arithmetic progression $P_A$ containing $A$ has size exactly $|A+A|-|A|+1\le 2|A|-3$. This description involves the notion of additively left-stable sets.  \\
A bounded subset $X$ of $\R$ or $\Z$ will be {\it additively left-stable} if $\inf(X)=0$ and $(X+X)\cap[0,\sup(X)]=X$.\\
Freiman states that critical sets $A\subset \Z$ such that the smallest arithmetic progression $P_A$ containing $A$ has size exactly $|A+A|-|A|+1\le 2|A|-3$ are a partition of three sets of the form $A=(\min(A)+A_1)\cup P\cup (\max(A)-A_2)$ where $A_1 $ and $A_2$ are additively left-stable sets and $P$ is an arithmetic progression. In this paper we shall give the maximum density of left stable sets of integers on intervals $[0,x]$.

To this purpose, we use a theorem by Grynkiewicz which improves on Theorem \ref{3k-4} when we have some information on the ratio between the size of $A$ and the size of $B$. From the weaker hypothesis $|A+B|\le |A|+s|B|-2s$, with some $s$ depending on the relative sizes of $A$ and $B$, Grynkiewicz gets the conclusion $B\subset P_B$ with $P_B$ an arithmetic progression of size $|P_B|\le |A+B|-|A|+1$.\\
\\
{Grynkiewicz's} theorem is actually a discrete version of a result stated by Ruzsa in 1991 \cite{Ruz}. 
\begin{thm}[Ruzsa]\label{ruz}
Let $A$ and $B$ be two bounded subsets of $\R$ of positive inner Lebesgue measure. 
If $K$ is the unique integer such that 
$$\frac{K(K-1)}2\le \frac{\lambda(A)}{\lambda(B)}<\frac{K(K+1)}2,$$
then either\\
 \begin{equation}\label{min_sum1}
 \lambda(A+B)\geq \lambda(A)+{\rm diam}(B) 
\,\mbox{
where }\,{\rm diam}(B) =\sup(B)-\inf(B)
\end{equation}
or
 \begin{equation}\label{min_sum2}
 \lambda(A+B)\geq (K+1)\left(\frac{\lambda(A)}{K}+\frac{\lambda(B)}2\right).
 \end{equation}
 \end{thm}
 
 \begin{rem}
 The inequality \eqref{min_sum2} implies $ \lambda(A+B)\geq \lambda(A)+K\lambda(B)$.
 \end{rem}
 
 In \cite{JEP}, the second author deduced from this theorem a continuous analog of the $3k-4$ theorem (Theorem \ref{3k-4}), completely characterized the critical sets $A$ and $B$ such that equality holds in \eqref{min_sum2} and gave a description of critical sets such that equality holds in \eqref{min_sum1}. 
In the symmetric case $A=B$ these last critical sets can be partitioned into three parts as follows 
$$A=\min(A)+\left(A_1\cup I\cup ({\rm diam}(A)-A_2)\right)$$ where $I$ is an interval of measure $\lambda(I)=2\lambda(A)-{\rm diam}(A)$ and $A_1$ and $A_2$ are, up to a set of measure $0$, additively left-stable sets such that $\lambda(A_i)=\frac12{\rm diam}(A_i)$.
\\
The property of being additively left-stable forces the density of the set near $0$ to be small. This statement will be made more precise in Theorem \ref{thmstabA} where we find a function $h$ such that if $d={\rm diam}(A)$, then for any $x\in [0,d]$, $\lambda(A\cap[0,x])\le d h(x/d)$ if $A$ is left stable of density $d/2$. We also prove that this upper bound is sharp and describe for any fixed $x\in (0,1)$, a set $A$  of density $1/2$ on $[0,1]$ such that $\inf(A)=0$, $\sup(A)=1$, $(A+A)\cap[0,1]=A$ and $\lambda(A\cap[0,x])= h(x)$. \\
\\
The section 2 of this paper will be devoted to the study of additively left stable subsets of real numbers. 
Our results in the continuous setting will be extended to the discrete setting in section 3. \\
\\
If $k$ is a positive integer, $kA$ will denote the Minkowski sum of $k$ copies of $A$. 
   
 \section*{Acknowledgement}

The authors would like to express their gratitude to Catherine Jaulent, a wonderful teacher, for cultivating their passion for mathematics. 
They also would like to thank the referee for his suggestions which helped them to clarify some arguments.

\section{Sumsets in $\R$, additively left stable sets}
In this section, we shall write $\lambda$ for the inner Lebesgue measure. The inner Lebesgue measure $\lambda(A)$ of a set $A$ is defined as the supremum of the Lebesgue measure of closed sets included in $A$. We use the inner measure to avoid any problem of measurability.\\    
In Theorem 3 of \cite{JEP}, the second author gave a description of critical sets for Ruzsa's inequality \eqref{min_sum1}. In the symmetric case, her result can be stated this way.
\begin{thm}\label{critA_JEP}
Let $A$ be a bounded closed set of real numbers of positive inner Lebesgue measure.\\
Assume that $\lambda(A+A)={\rm diam}(A)+\lambda(A)<3\lambda(A)$. \\
Then $A$ can be partitioned into three parts as follows 
$$A=\min(A)+\left(A_1\cup I\cup ({\rm diam}(A)-A_2)\right)$$ where 
\begin{itemize}
\item $I$ is an interval of measure $\lambda(I)= 2\lambda(A)-{\rm diam}(A)$,
\item for $i=1,2$, $\min(A_i)=0$, $\lambda(A_i)=\frac12{\rm diam}(A_i)$ and $(A_i+A_i)\cap[0,{\rm diam}(A_i)]=A_i$ up to a set of measure $0$.
\end{itemize}
\end{thm}

We shall say that a nonempty bounded subset $X$ of $\R$ is {\it additively left stable} if $\inf(X)=0$ and $(X+X)\cap[0,{\rm diam}(X)]=X$. We say that $X$ {\it has density $1/2$} if $\lambda(X)=\frac12{\rm diam}(X)$.\\
\\
Note that if $X$ is an additively left-stable set of diameter $d$, then for any positive integer $k$, we have $kX\cap[0,d]=X$. If $\lambda(X)\not=d$ (i.e. if $X$ is not of full measure in an interval), this forces the density of the set near $0$ to be small. 
One of the goals of this section is to make this statement more precise.

\begin{lem}\label{lem_stab}
Let $A$ be an additively left-stable bounded subset of real numbers. \\
We write $d={\rm diam}(A)$ and for $x\in\R$, $A_x=A\cap[0,x]$. \\
Then for any $x\in A$, for any integer $k\ge 2$ such that $\lambda(A\cap[(k-1)x,d])<\lambda([(k-1)x,d])$, we have
$$\lambda(A_x)\le \begin{cases} 
\frac2{k+1}\left(x-\frac{d-\lambda(A)}k\right)&\mbox{ if } kx>d, \\
\frac2{k(k+1)}\lambda(A_{kx})&\mbox{ if } kx\leq d. 
\end{cases}
.$$
\end{lem}
\begin{rem}
The hypothesis $\lambda(A\cap[(k-1)x,d])<\lambda([(k-1)x,d])$ implies $(k-1)x<d$.
\end{rem}

\begin{proof}
We can assume without loss of generality that $A$ is a closed set, thus $0\in A$. To get the general statement, we then apply the result to an increasing sequence of closed sets $A_n$ ($A_n\subset A_{n+1}$) included in $A$ such that $(\lambda(A_n))_n$ tends to $\lambda(A)$ and $({\rm diam}(A_n))_n$ tends to $d$.\\ 
We write $c=\inf\left\{y\in [0,d]\,:\,\lambda(A\cap[y,d])=d-y\right\}$ and $A'=A\cup[c,+\infty[$. 
Note that $A'+A'=A'$ and
$$\begin{cases}
\forall y\ge 0,&\lambda\left((A+A)_y\right)\leq\lambda(A'_y)\\ \forall y\in[0,d],& \lambda\left(A_y\right)=\lambda(A'_y).
\end{cases}$$
Let $x\in A$ be such that $\lambda(A_x)>0$ (if $\lambda(A_x)=0$, the expected inequality is trivial). \\
\\
For $k\ge 2$, the hypothesis $\lambda(A\cap[(k-1)x,d])<\lambda([(k-1)x,d])$ is equivalent to $(k-1)x< c$, which implies $x<c$, so we assume $x<c$. 
Let $\ell\ge 2$ be an integer such that $(\ell-1)x<c\le d$.\\ 
The inequality $\lambda(A_{(\ell-1) x}+A_x)\ge\lambda(A_{(\ell-1) x})+x$ would imply $\lambda(A'\cap[(\ell-1)x,\ell x])=x$.\\
Since $0,x\in A$, we would have 
\begin{align*}
\lambda(A'\cap[(\ell-1)x,(\ell+1) x])&=\lambda((A'+A')\cap[(\ell-1)x,(\ell+1) x])\\
&\ge \lambda(0+(A'\cap[(\ell-1)x,\ell x]))+\lambda(x+(A'\cap[(\ell-1)x,\ell x]))=2x
\end{align*}
and iterating, we get $\lambda(A\cap[(\ell-1)x,d])=\lambda(A'\cap[(\ell-1)x,d])=\lambda([(\ell-1)x,d])$, a contradiction to the minimality of $c$. \\
\\
Therefore, for any $\ell\ge 2$ such that $(\ell-1)x<c$, we have $$\lambda(A_{(\ell-1) x}+A_x)<\lambda(A_{(\ell-1) x})+{\rm diam}(A_x).$$
This is also true for $\ell=1$ since we assumed $x<c$.\\
We shall prove by induction on $k\ge 1$ such that $(k-1)x<c$ that we have $$\lambda(A'_{k x})\ge\frac{k(k+1)}{2}\lambda(A_x).$$
This is obviously true for $k=1$.\\
Let $k\ge 2$ be an integer such that $(k-1)x<c$ and $\lambda(A'_{(k-1) x})\ge\frac{k(k-1)}{2}\lambda(A_x).$\\
We proved that we must have $\lambda(A_{(k-1) x}+A_x)<\lambda(A_{(k-1) x})+{\rm diam}(A_x)$. Using Ruzsa's inequality stated in Theorem \ref{ruz} and $\lambda(A_{(k-1) x})\ge\frac{k(k-1)}{2}\lambda(A_x)$ leads to
$$\lambda(A'_{k x})\geq\lambda(A_{(k-1) x}+A_x)\ge \lambda(A_{(k-1) x})+k\lambda(A_x)\ge\frac{k(k+1)}{2}\lambda(A_x).$$
If $kx\le d$, then $\lambda(A'_{k x})=\lambda(A_{kx})$ so
$$\lambda(A_x)\leq \frac2{k(k+1)}\lambda(A_{kx}).$$
Otherwise, $\lambda(A'_{k x})=\lambda(A)+kx-d$ so
$$\lambda(A_x)\leq \frac2{k(k+1)}\left(\lambda(A)+kx-d\right).$$
\end{proof}

\begin{thm}\label{thmstabA}
Let $A$ be an additively left-stable bounded subset of real numbers. \\
We write $d={\rm diam}(A)$ and for $x\in\R$, $A_x=A\cap[0,x]$. \\
If $\lambda(A)=d/2>0$, then for all $x\in[0,d]$, $\lambda(A_x)\le dh(x/d)$ where $h$ (see figure 1) is the continuous, piecewise affine function defined on $[0,1]$ by 
$$h(x)=
\begin{cases}
0,&\mbox{ if }x=0,\\
\frac1{k+1}\left(2x-\frac{1}k\right),&\mbox{ if } x\in\left(\frac{1}k,\frac{1}{k-1}\right],\, k\ge 2.
\end{cases}$$
\end{thm}

\begin{rem}
The hypothesis $\lambda(A)=d/2$ allows us to get a neater result. We could get rid of this hypothesis by noticing that :
\\$\bullet$ if $\lambda(A)<d/2$ then we can take $d'=2d-2\lambda(A)$ and $A'=A\cup[d,d']$ satisfies the hypothesis of Theorem \ref{thmstabA} with ${\rm diam}(A')=d'$, which implies $\lambda(A'_x)\le d'h(x/d')$.
\\$\bullet$ if $\lambda(A)>d/2$ then we can take $d'\in[0,d]$ such that $\lambda(A_{d'})=d'/2$, $A=A'\cup[d',d]$ with $A'=A\cap[0,d']$ satisfying the hypothesis of Theorem \ref{thmstabA} with ${\rm diam}(A')=d'$, which implies $\lambda(A'_x)\le d'h(x/d')$ for $x\le d'$.\\
In this last case, the inclusion $[d',d]\subset A$ is due to the stability of $A$ and the inequality $\lambda(A+A)\ge d+\lambda(A)$. 
\end{rem}

\begin{center}
\begin{tikzpicture}[xscale=10,yscale=10]
\draw [<->] (0,0.6) -- (0,0) -- (1.1,0);
\node [below left] at (0,0) {\small $0$};
\draw [dashed] (0,0.5) -- (1,0.5);
\node [left] at (0,0.5) {\small $1/2$};
\draw [dashed] (1,0) -- (1,0.5);
\node [below] at (1,0) {\small $1$};
\draw [dashed] (0.5,0) -- (0.5,1/6);
\draw [dashed] (0,1/6) -- (0.5,1/6);
\node [below] at (0.5,0) {\small $1/2$};
\node [left] at (0,1/6) {\small $1/6$};
\draw[fill,red] (1,0.5) circle [radius=0.003] ;
\draw[fill,red] (0.5,1/6) circle [radius=0.003] ;
 \draw[fill,red] (1/3,1/12) circle [radius=0.002] ;
\draw[fill,red] (1/4,1/20) circle [radius=0.002] ;
\draw[fill,red] (1/5,1/30) circle [radius=0.003] ;
\draw [dashed] (1/5,0) -- (1/5,1/30);
\draw [dashed] (0,1/30) -- (1/5,1/30);
\node [below] at (1/5,0) {\small $1/k$};
\node [left] at (0,1/30) {\small $1/(k(k-1))$};
\draw[fill,red] (1/6, 1/42) circle [radius=0.002] ;
\draw[fill,red] (1/7,1/56) circle [radius=0.002] ;
\draw[fill,red] (1/8,1/72) circle [radius=0.002] ;
\draw[fill,red] (1/9, 1/90) circle [radius=0.002] ;
\draw[fill,red] (1/10,1/110) circle [radius=0.002] ;
\draw[fill,red] (1/11, 1/132) circle [radius=0.001] ;
\draw[fill,red] (1/12, 1/156) circle [radius=0.001] ;
\draw[fill,red] (1/13, 1/182) circle [radius=0.001] ;
\draw[fill,red] (1/14, 1/210) circle [radius=0.001] ;
\draw[fill,red] (1/15, 1/240) circle [radius=0.001] ;
\draw[fill,red] (1/16, 1/272) circle [radius=0.001] ;
\draw[fill,red] (1/17, 1/306) circle [radius=0.001] ;
\draw[fill,red] (1/18, 1/342) circle [radius=0.001] ;

\draw[red](1,0.5) -- (0.5,1/6);
\node[red, above] at (0.75,0.34) {$h$};
\draw[red](0.5,1/6) -- (1/3,1/12);
\draw[red](1/4,1/20) -- (1/3,1/12);
\draw[red](1/4,1/20) -- (1/5,1/30);
\draw[red](1/5,1/30) -- (1/6,1/42);
\draw[red](1/7,1/56) -- (1/6,1/42);
\draw[red](1/7,1/56) -- (1/8,1/72);
\draw[red](1/9,1/90) -- (1/8,1/72);
\draw[red](1/9,1/90) --  (1/10,1/110);
\draw[red] (1/10,1/110)-- (1/11,1/132);
\draw[red] (1/12,1/156)-- (1/11,1/132);
\draw[red] (1/12,1/156)-- (1/13,1/182);
\draw[red] (1/14,1/210)-- (1/13,1/182);

\node[below] at (1/2,-0.08) {Figure 1: graph of function $h$.}; 
\end{tikzpicture}
\end{center}

\begin{proof}
Let $k\ge 2$ be an integer and let $x\in A\cap \left(\frac{d}k,\frac{d}{k-1}\right]$. Since $kx> d$ and $\lambda(A)=d/2$, Lemma \ref{lem_stab} gives the conclusion when $\lambda(A\cap[(k-1)x,d])<\lambda([(k-1)x,d])$.
\\
If $\lambda(A\cap[(k-1)x,d])=\lambda([(k-1)x,d])$, then with the notation used in the proof of Lemma \ref{lem_stab}, we have $c\le (k-1)x$. \\
If $x< c$, let $\ell$ be the largest integer such that $(\ell-1) x<c\le d$ ($\ell\ge 2$). We have $\ell<k$ thus $\ell x\le(k-1)x\le d$ and $\ell x\ge c$ thus $\lambda(A_{\ell x})=\lambda(A)-(d-\ell x)$. Therefore, by Lemma \ref{lem_stab},
$$\ell x-\frac{d}2=\lambda(A)-(d-\ell x)=\lambda(A_{\ell x})\ge\frac{\ell(\ell+1)}{2}\lambda(A_x).$$
Let $(u_i)_{i\ge 1}=(u_i(x,d))_{i\ge 1}$ be the sequence defined by 
\begin{equation}\label{def_ui}
u_i=u_i(x,d)=\frac1{i+1}(2x-d/i)
\end{equation}
We have $u_{i-1}\le u_i\Leftrightarrow (i-1)x\le d$ and $u_i=u_{i-1}\Leftrightarrow (i-1)x=d$. Since $(k-1)x\le d<kx$, the sequence $(u_i)_{1\le i\le k}$ is nondecreasing, the sequence $(u_i)_{i\ge k}$ is decreasing and for all $i\in\N$, $u_k\ge u_i$, this inequality being strict unless $i=k-1$ and $(k-1)x=d$.\\
It yields
$$\lambda(A_x)\le \frac1{\ell+1}\left(2 x-\frac{d}\ell\right)=u_\ell\le u_k=\frac1{k+1}\left(2x-\frac{d}k\right)=dh(x/d). $$
If on the other side $x\ge c$, then 
$$\lambda(A_x)=\lambda(A)-(d-x)=x-\frac{d}2=u_1\le u_k=\frac1{k+1}\left(2x-\frac{d}k\right).$$
In any case, we get the expected inequality when $x\in A$.\\ 
If $x\in A^c\cap \left(\frac{d}k,\frac{d}{k-1}\right]$ and $\lambda(A_x)>0$, let $x'\in(A\cap [0,x])$ and $k'$ be such that $x'\in  \left(\frac{d}{k'},\frac{d}{k'-1}\right]$. Then, using $u_k=\max_{i\in \N}u_i$ 
we get
$$\lambda(A_{x'})\le \frac1{k'+1}\left(2x'-\frac{d}{k'}\right)\le \frac1{k'+1}\left(2x-\frac{d}{k'}\right)=u_{k'}\le u_k=\frac1{k+1}\left(2x-\frac{d}k\right).$$
If $x'$ tends to $\sup(A\cap [0,x])$, we get the expected inequality.
\end{proof}

The following proposition proves that the upper bound in Theorem \ref{thmstabA} is sharp, that for a given set $A$ satisfying the hypothesis of Theorem  \ref{thmstabA}, the upper bound can only be sharp in one single point in $(0,d)$   and that for each $x\in(0,d)$ there exist a unique closed set satisfying this sharpness at the point $x$.
\begin{thm}
Let $d>0$ be a real number. For any additively left-stable bounded subset  of real numbers $A\subset [0,d]$ of diameter $d$  such that $\lambda(A)=d/2$, there exists at most one point $x\in(0,d)$ such that $\lambda(A_x)= dh(x/d)$. \\
Furthermore, for all $x\in(0,d)$ there exists a unique additively left-stable closed set $A$ of diameter $d$ such that $\lambda(A)=d/2$ and $\lambda(A_x)=d h(x/d)=\frac1{k+1}\left(2x-\frac{d}k\right)$, where $k$ is the unique integer such that $(k-1)x<d\le kx$. This set is the following union of intervals:
\begin{equation}\label{sharpset}
\bigcup_{i=0}^{k-1}\left[i\left(\frac{k-1}{k+1}x+\frac{d}{k(k+1)}\right),ix\right]\bigcup\left[\frac{k(k-1)}{k+1}x+\frac{d}{{k+1}},d\right].
\end{equation}
 \end{thm}
 
 \begin{rem}
 In other words, if $\mathcal{E}$ is the set of additively left-stable bounded subsets $A$ of real numbers such that $\lambda(A)=d/2$ and $\phi:(0,d)\rightarrow\mathcal{E}$ is the function defined by  $\phi(x)= \eqref{sharpset} $, then for any $A\in\mathcal{E}$, the set $P$ of elements $y\in(0,1)$ such that $\lambda(A_y)=dh(y/d)$ is empty unless there exist $x\in(0,1)$ such that $A=\phi(x)$ in which case $P=\{x\}$ ($\phi$ is a one to one function). 
 \end{rem}
 
 The following figure shows how the set $A=\phi(x)$ is obtained and the graph of the function $y\mapsto\lambda(A_y)$.
 
 \begin{center}
\begin{tikzpicture}[xscale=10,yscale=10]
\draw [<->] (0,0.6) -- (0,0) -- (1.1,0);
\node [below left] at (0,0) {\small $0$};
\draw [dashed] (0,0.5) -- (1,0.5);
\node [left] at (0,0.5) {\small $d/2$};
\draw [dashed] (1,0) -- (1,0.5);
\node [below] at (1,0) {\small $d$};
\draw [dashed] (0.5,0) -- (0.5,1/6);
\draw [dashed] (0,1/6) -- (0.5,1/6);
\node [below] at (0.5,0) {\small $d/2$};
\node [left] at (0,1/6) {\small $d/6$};
\draw[fill,red] (1,0.5) circle [radius=0.003] ;
\draw[fill,red] (0.5,1/6) circle [radius=0.003] ;
 \draw[fill,red] (1/3,1/12) circle [radius=0.002] ;
\draw[fill,red] (1/4,1/20) circle [radius=0.002] ;
\draw[fill,red] (1/5,1/30) circle [radius=0.003] ;
\draw [dashed] (1/5,0) -- (1/5,1/30);
\draw [dashed] (0,1/30) -- (1/5,1/30);
\node [below] at (1/5,0) {\small $d/k$};
\node [left] at (0,1/30) {\small $d/(k(k-1))$};
\draw[fill,red] (1/6, 1/42) circle [radius=0.002] ;
\draw[fill,red] (1/7,1/56) circle [radius=0.002] ;
\draw[fill,red] (1/8,1/72) circle [radius=0.002] ;
\draw[fill,red] (1/9, 1/90) circle [radius=0.002] ;
\draw[fill,red] (1/10,1/110) circle [radius=0.002] ;
\draw[fill,red] (1/11, 1/132) circle [radius=0.001] ;
\draw[fill,red] (1/12, 1/156) circle [radius=0.001] ;
\draw[fill,red] (1/13, 1/182) circle [radius=0.001] ;
\draw[fill,red] (1/14, 1/210) circle [radius=0.001] ;
\draw[fill,red] (1/15, 1/240) circle [radius=0.001] ;
\draw[fill,red] (1/16, 1/272) circle [radius=0.001] ;
\draw[fill,red] (1/17, 1/306) circle [radius=0.001] ;
\draw[fill,red] (1/18, 1/342) circle [radius=0.001] ;

\draw[red](1,0.5) -- (0.5,1/6);
\node[red, above] at (0.75,0.34) {$h$};
\node[blue, below] at (0.75,0.24) {$y\mapsto \lambda(A_y)$};
\draw[red](0.5,1/6) -- (1/3,1/12);
\draw[red](1/4,1/20) -- (1/3,1/12);
\draw[red](1/4,1/20) -- (1/5,1/30);
\draw[red](1/5,1/30) -- (1/6,1/42);
\draw[red](1/7,1/56) -- (1/6,1/42);
\draw[red](1/7,1/56) -- (1/8,1/72);
\draw[red](1/9,1/90) -- (1/8,1/72);
\draw[red](1/9,1/90) --  (1/10,1/110);
\draw[red] (1/10,1/110)-- (1/11,1/132);
\draw[red] (1/12,1/156)-- (1/11,1/132);
\draw[red] (1/12,1/156)-- (1/13,1/182);
\draw[red] (1/14,1/210)-- (1/13,1/182);

%
 \node [below ] at (0,-0.10){\tiny $0$};
  \draw []  (1,-0.11) -- (1,-0.09);

  \draw[fill,blue] (1/3+0.05,1/12+0.025 ) circle [radius=0.002] ;
  \draw[dashed, blue](1/3+0.05,0) -- (1/3+0.05,1/12+0.025);
\draw[blue] (1/4+0.025,0) -- (1/3+0.05, 1/12+0.025) -- (1/2+0.05,1/12+0.025) -- (2/3+0.1, 1/4+0.075) -- (3/4+0.075,1/4+0.075)--(1,1/2);
\node [below, blue] at (1/3+0.05,0){$x$};
 \draw []  (0,-0.1) -- (1,-0.1);
 \draw []  (0,-0.11) -- (0,-0.09);
  \draw []  (1,-0.11) -- (1,-0.09);
 \node [below ] at (1,-0.10){\tiny $1$};
  \draw[fill,blue] (0,-0.1 ) circle [radius=0.003] ;

\draw[blue,line width=1] (1/4+0.025,-0.1) -- (1/3+0.05,-0.1);
\node [below, blue ] at (0.33,-0.10){\tiny $I_1$};
\draw[blue, line width=1] (1/2+0.05,-0.1) -- (2/3+0.1,-0.1); 
\node [below, blue ] at (0.67,-0.10){\tiny $I_2$};
\draw[blue, line width=1] (3/4+0.075,-0.1) -- (1,-0.1); 
\node [below, blue ] at (0.9,-0.10){\tiny $I_3$};
\node [left, blue] at (0,-0.1){$A=\phi(x)$};
\node at (1/2,-0.18){Figure 2: construction of the set $\phi(x)$.};
\end{tikzpicture}
\end{center}

\begin{proof}
For each $x\in(0,d)$, let $k$ be the integer such that $x\in\left[\frac{d}k,\frac{d}{k-1}\right)$.\\
Define $A=\cup_{i=0}^{k}I_i$ with $I_i=\left[i\left(\frac{k-1}{k+1}x+\frac{d}{k(k+1)}\right),ix\right]$ for $i\le k-1$ and $I_k=\left[\frac{k(k-1)}{k+1}x+\frac{d}{{k+1}},d\right]$.
The intervals $I_i$ are pairwise disjoints. Indeed, since $x<\frac{d}{k-1}$, we have for  $i\le k-1$,
\[\frac{i+1}{k}d\ge\frac{2i-(k-1)}{k-1}d>(2i-(k-1))x, \]
thus
\[(i+1)\left(\frac{k-1}{k+1}x+\frac{d}{k(k+1)}\right)>ix.\]

%
Furthermore, each interval $I_i$ is nonempty 
so
\begin{align*}
\lambda(A)&=\sum_{i=0}^{k-1}i\left(x-\left(\frac{k-1}{k+1}x+\frac{d}{k(k+1)}\right)\right)+\left(d-\frac{k(k-1)}{k+1}x-\frac{d}{{k+1}}\right)\\
&=\frac{k(k-1)}2\left(\frac{2}{k+1}x-\frac{d}{k(k+1)}\right)+\frac{kd}{k+1}-\frac{k(k-1)}{k+1}x=\frac{d}2.
\end{align*}
Now,
$$(I_i+I_j)\cap[0,d]=\begin{cases}I_{i+j} &\mbox{ if } i+j\le k\\ 
\emptyset&\mbox{ if } i+j\ge k+1\end{cases}$$
so $(A+A)\cap [0,d]=A$.\\
Furthermore, $\lambda(A_x)=\left(\frac{2}{k+1}x-\frac{d}{k(k+1)}\right)$.\\
It proves that the set $A$ satisfies the required properties.\\

\smallskip

It remains to prove the unicity of such a set.
Let $A$ be a closed set of diameter $d$ and minimum $0$ such that $(A+A)\cap [0,d]=A$, $\lambda(A)=d/2$ and $\lambda(A_x)= \frac1{k+1}\left(2x-\frac{d}k\right)$, where $k$ is the unique integer such that $(k-1)x<d\le kx$.

If there exists $\ell\le k-1$ such that $\lambda(A\cap[\ell x,d])=\lambda([\ell x,d])$, then taking $\ell$ minimal, we have $\lambda(A\cap[(\ell-1)x,d])<\lambda([(\ell-1)x,d])$ and Lemma \ref{lem_stab} implies 
$$\lambda(A_x)\le\frac2{\ell(\ell+1)}\lambda(A_{\ell x})=\frac2{\ell(\ell+1)}\left(\ell x-\frac{d}2\right)=\frac1{\ell+1}\left(2 x-\frac{d}\ell\right)=u_\ell.$$
Using $u_k>\max_{i\not=k}u_i$,
we get
$$\lambda(A_x)\le u_\ell<u_k=\frac1{k+1}\left(2 x-\frac{d}{k}\right),$$
in contradiction with the hypotheses. Therefore $\lambda(A\cap[(k-1) x,d])<\lambda([(k-1) x,d])$ and Lemma \ref{lem_stab} yields $\lambda(A_{(k-1)x})\ge\frac{k(k-1)}2\lambda(A_x)$. Using Ruzsa's inequality  \eqref{min_sum2} and going through the proof of Lemma \ref{lem_stab} with the same notation we get 
\begin{align*}
kx-\frac{d}2=\lambda(A_{kx}')&\ge\lambda(A_{(k-1)x}+A_x)\\
&\ge \lambda(A_{(k-1)x})+k\lambda(A_x)\\&\ge\frac{k(k-1)}2\lambda(A_x)+k\lambda(A_x)\\
& \ge\frac{k(k+1)}2\lambda(A_x)\\
&\ge\frac{k(k+1)}2\left(\frac{2}{k+1}x-\frac{d}{k(k+1)}\right)=kx-\frac{d}2.
\end{align*}
Therefore these inequalities are equalities and
\[\lambda(A_{(k-1)x}+A_x)=\lambda(A_{(k-1)x})+k\lambda(A_x),\] with \[\lambda(A_{(k-1) x})=\frac{k(k-1)}2\lambda(A_x,)\]
and the sets $A_x$ and $A_{(k-1)x}$ are critical sets for \eqref{min_sum2}.
According to Theorem 5 of \cite{JEP}, the set $A_x$ is the union of two intervals $[0,a]\cup[b,{\rm diam}(A_x)]$. By Theorem \ref{thmstabA}, $\lambda(A_a)\le dh(a/d)<a$ if $a>0$, which implies $a=0$ and $A_x=\{0\}\cup[b,{\rm diam}(A_x)]$.\\
Since $dh(x/d)$ is extremal, $x$ has to be the diameter of $A_x$. 
Therefore $A_x=\{0\}\cup I_1$ since $\lambda(A_x)=\left(\frac{2}{k+1}x-\frac1{k(k+1)}\right)$.\\
By stability of $A$, $\lambda(A)=d/2$ and using that $A$ is closed, we get the expected expression of $A$.\\
Given that for $y\in(0,d)\setminus\{x\}$, $\lambda(A_y)<dh(y/d)$, the first statement of the theorem is also proved.  
\end{proof}

Combining Theorem \ref{thmstabA} with Theorem \ref{critA_JEP}, we get upper bounds for the density near the borders of critical sets for Ruzsa's inequality \eqref{min_sum1} in the symmetric case.
\begin{cor}\label{critA_new}
Let $A$ be a bounded closed set of real numbers of positive inner Lebesgue measure and diameter $d$.\\
Assume that $\lambda(A+A)={\rm diam}(A)+\lambda(A)<3\lambda(A)$. \\
Then there exist $b\in[0,d]$ such that the function $g(x)=\lambda((A-\inf(A))\cap[0,x])$ is a non negative, non decreasing continuous function of the form
$$g(x)=\begin{cases}g_1(x)&\mbox{ if }x\in [0,b],\\x-\frac{b}2&\mbox{ if }x\in [b,b+\Delta],\\g_2(x)&\mbox{ if }x\in [b+\Delta,{\rm diam}(A)]\end{cases}$$
with $\Delta=2\lambda(A)-{\rm diam}(A)$, $g_1(x)\le bh(x/b)$ and $g_2(x)\ge \lambda(A)-(d-\Delta-b)h\left(\frac{d-x}{d-\Delta-b}\right)$.
\end{cor}

\begin{center}
\begin{tikzpicture}[xscale=6,yscale=6]
\draw [<->] (0,1.3) -- (0,0) -- (2.1,0);
\node [below left] at (0,0) {\small $0$};
\draw [dashed] (0,0.5) -- (1,0.5);
\node [left] at (0,0.5) {\small $b/2$};
\draw [dashed] (1,0) -- (1,0.5);
\node [below] at (1,0) {\small $b$};
\draw [dashed] (0.5,0) -- (0.5,1/6);
\draw [dashed] (0,1/6) -- (0.5,1/6);
\node [below] at (0.5,0) {\small $b/2$};
\node [left] at (0,1/6) {\small $b/6$};
\draw[fill,red] (1,0.5) circle [radius=0.003] ;
\draw[fill,red] (0.5,1/6) circle [radius=0.003] ;
 \draw[fill,red] (1/3,1/12) circle [radius=0.002] ;
\draw[fill,red] (1/4,1/20) circle [radius=0.002] ;
\draw[fill,red] (1/5,1/30) circle [radius=0.003] ;
\draw [dashed] (1/5,0) -- (1/5,1/30);
\draw [dashed] (0,1/30) -- (1/5,1/30);
\node [below] at (1/5,0) {\small $b/k$};
\node [left] at (0,1/30) {\small $b/(k(k-1))$};
\draw[fill,red] (1/6, 1/42) circle [radius=0.002] ;
\draw[fill,red] (1/7,1/56) circle [radius=0.002] ;
\draw[fill,red] (1/8,1/72) circle [radius=0.002] ;
\draw[fill,red] (1/9, 1/90) circle [radius=0.002] ;
\draw[fill,red] (1/10,1/110) circle [radius=0.002] ;
\draw[fill,red] (1/11, 1/132) circle [radius=0.001] ;
\draw[fill,red] (1/12, 1/156) circle [radius=0.001] ;
\draw[fill,red] (1/13, 1/182) circle [radius=0.001] ;
\draw[fill,red] (1/14, 1/210) circle [radius=0.001] ;
\draw[fill,red] (1/15, 1/240) circle [radius=0.001] ;
\draw[fill,red] (1/16, 1/272) circle [radius=0.001] ;
\draw[fill,red] (1/17, 1/306) circle [radius=0.001] ;
\draw[fill,red] (1/18, 1/342) circle [radius=0.001] ;

\draw[red](1,0.5) -- (0.5,1/6);
\draw[red](0.5,1/6) -- (1/3,1/12);
\draw[red](1/4,1/20) -- (1/3,1/12);
\draw[red](1/4,1/20) -- (1/5,1/30);
\draw[red](1/5,1/30) -- (1/6,1/42);
\draw[red](1/7,1/56) -- (1/6,1/42);
\draw[red](1/7,1/56) -- (1/8,1/72);
\draw[red](1/9,1/90) -- (1/8,1/72);
\draw[red](1/9,1/90) --  (1/10,1/110);
\draw[red] (1/10,1/110)-- (1/11,1/132);
\draw[red] (1/12,1/156)-- (1/11,1/132);
\draw[red] (1/12,1/156)-- (1/13,1/182);
\draw[red] (1/14,1/210)-- (1/13,1/182);

\draw[blue](1,0.5) -- (0.5,0);

\draw[violet](1,0.5) -- (1.3,0.8);
\draw[blue](1.5,1) -- (1.3,0.8);
\draw[blue](1.5,1) -- (1.7,1);
\draw [dashed] (1.3,0) -- (1.3,0.8);
\draw [dashed] (0,0.8) -- (1.3,0.8);
\node[below]at(1.3,0) {$b+\Delta$};
\node[left]at(0,0.8) {$\frac{b}2+\Delta$};
\draw [dashed] (1.7,0) -- (1.7,1);
\draw [dashed] (0,1) -- (1.7,1);
\node[below]at(1.7,0) {$d$};
\node[left]at(0,1) {$\lambda(A)$};

\draw[fill,red] (1.5,1-0.4/6) circle [radius=0.003] ;
\draw[red] (1.5,1-0.4/6)  -- (1.3,0.8);

\draw[fill,red] (1.7-0.4/3,1-0.4/12) circle [radius=0.003] ;
\draw[red] (1.7-0.4/3,1-0.4/12)  -- (1.5,1-0.4/6) ;

\draw[fill,red] (1.7-0.4/4,1-0.4/20) circle [radius=0.003] ;
\draw[red] (1.7-0.4/3,1-0.4/12)  -- (1.7-0.4/4,1-0.4/20);
\draw[fill,red] (1.7-0.4/5,1-0.4/30) circle [radius=0.003] ;
\draw[red] (1.7-0.4/5,1-0.4/30)  -- (1.7-0.4/4,1-0.4/20);
\draw[fill,red] (1.7-0.4/6,1-0.4/42) circle [radius=0.003] ;
\draw[red] (1.7-0.4/5,1-0.4/30)  -- (1.7-0.4/6,1-0.4/42);
\draw[fill,red] (1.7-0.4/7,1-0.4/56) circle [radius=0.003] ;
\draw[red] (1.7-0.4/7,1-0.4/56)  -- (1.7-0.4/6,1-0.4/42);
\draw[fill,red] (1.7-0.4/8,1-0.4/72) circle [radius=0.003] ;
\draw[red] (1.7-0.4/7,1-0.4/56)  -- (1.7-0.4/8,1-0.4/72);
\node[below] at(1,-0.1) {Figure 3: bounds for the density of critical sets for Ruzsa's inequality.};
\end{tikzpicture}
\end{center}

\begin{rem}
Corollary \ref{critA_new} states that the graph of the function $g$ is between the red graph and the blue one in figure 3.
\end{rem}

\begin{proof}
According to Theorem  \ref{critA_JEP}, 
$$A=\min(A)+\left(A_1\cup I\cup (d-A_2)\right)$$ where 
\begin{itemize}
\item $I$ is an interval of measure $\Delta= 2\lambda(A)-d$,
\item for $i=1,2$, $\inf(A_i)=0$, $\lambda(A_i)=\frac12{\rm diam}(A_i)$ and $\lambda((A_i+A_i)\cap[0,{\rm diam}(A_i)])=\lambda(A_i)$.
\end{itemize}
Write $b=\inf I$. Then we have ${\rm diam}(A_1)=b$, $0,b\in A_1$ since $A_1$ is closed.\\
Since $\lambda((A_1+A_1)\cap[0,b])=\lambda(A_1)$, we claim that there exists a closed set $K\subset A_1$ such that $\lambda(K)=\lambda(A_1)$ and $(K+K)\cap[0,b]=K$.\\
We can define such a set $K$ by 
$$K=\{0,b\}\cup\left(\cap_{n\ge 1}K_n\right)\quad\mbox{ where }\quad K_n=\{x\in[0,b]\,:\, \lambda\left(A\cap\left[x-1/n,x+1/n\right]\right)>0\}.$$
Indeed, $K\subset A_1$ since $A_1$ is a closed set.\\
On the one side $K\subset (K+K)\cap[0,b]$ since $0\in K$ and on the other side if $x=y+z\in[0,b]$ with $y,z\in K\setminus\{0\}$ then for any positive integer $n$ such that $\left[x-1/n,x+1/n\right]\subset[0,b]$, $\lambda\left(A_1\cap\left[y-1/(2n),y+1/(2n)\right]\right)>0$ and $\lambda\left(A_1\cap\left[z-1/(2n),z+1/(2n)\right]\right)>0$ thus 
$$\lambda\left(A_1\cap\left[x-1/n,x+1/n\right]\right)=\lambda\left((A_1+A_1)\cap\left[x-1/n,x+1/n\right]\right)>0$$
and $x\in K$. This proves $K=(K+K)\cap[0,b]$.\\
Furthermore, if $x\in A_1\setminus K$, then there exists $n\in\N$ such that $\lambda([x-1/n,x+1/n]\cap A_1)=0$ which implies $\lambda (A_1)=\lambda(K)$.\\
Applying Theorem  \ref{thmstabA} to $K$ yields the result on $[0,b]$. 
Applying this to the set $-A$ we get the full statement of the theorem.
\end{proof}

\section{Additively left-stable sets in $\Z$}

In this section, we shall call {\it interval} a set of consecutive integers and if $a,b\in \N$ we shall write $[a,b]$ for the set $[a,b]\cap\Z$. We shall say that a bounded set $A$ of integers is in {\it normal form} if $\min(A)=0$ and $gcd(A)=1$. For simplicity, we shall state our results for sets of integers in normal form. The general statements can be deduced from this special case by translation ($B\mapsto B^*=B-\min(B)$) and division ($B^*\mapsto A=B^*/gcd(B^*)$). In the general case, arithmetic progressions replace intervals.

\smallskip

The symmetric version of Freiman's $3k-4$ theorem for sets $A$ in normal form states that given a nonempty finite subset $A$ of $\Z$ such that $\min(A)=0$ and $gcd(A)=1$, the inequality $|A+A|\le 3|A|-4$ implies that $A$ has diameter at most $|A+A|-|A|$ and that $A+A$ contains an interval of size at least $2|A|-1$. 

Freiman's result can be reformulated as follows: given a nonempty finite subset $A$ of $\Z$ in normal form, we have
\begin{equation}  
|A+A|\ge\min(3|A|-3,|A|+{\rm diam}(A)).
\end{equation}
Freiman characterized in \cite{Freiman73} the sets $A$ for which equality $|A+A|=3|A|-3$ holds (see also \cite{Jin15} for a corrected and more precise statement). He also gave in \cite{Freiman09} a general description of sets $A$ such that $|A+A|=|A|+{\rm diam}(A)<3|A|-3$. As in the continuous setting, this equality implies that $A$ can be partitioned into three parts, more precisely that
$$A=A_1\cup I\cup ({\rm diam}(A)-A_2)$$
with $I$ an interval of size at least $|I|\ge 2|A|-{\rm diam}(A)-2\ge 2$ and $A_1$ and $A_2$ additively left-stable sets in normal form where an {\it additively left stable set} is a subset $B$ of $\Z$ such that $\min(B)=0$ and $(B+B)\cap [0,{\rm diam}(B)]=B$.

As in the continuous setting, we shall give an upper bound for $|B\cap[0,x]|$ for additively left stable sets $B$. For this purpose, we need a discrete analog of Ruzsa's inequality. Such a result was given by Grynkiewicz in \cite{Gr20}.

Given $A$ and $B$ finite, non empty subsets of $\Z$ in normal form with $|B|\ge 3$, we take $s\ge 1$ as the unique integer such that
\begin{equation}\label{defs}
s(s-1)\left(\frac{|B|}2-1\right)+s-1<|A|\leq s(s+1)\left(\frac{|B|}2-1\right)+s.
\end{equation}
Grynkiewicz proved that if 
\begin{equation}\label{min_sum_disc}
 |A+B|< \left(\frac{|A|}s+\frac{|B|}2-1\right)(s+1),
 \end{equation}
 then ${\rm diam}(B)\le|A+B|-|A|$.
We will actually use the following weaker result.
\begin{prop}[Grynkiewicz]\label{corG}
Let $A$ and $B$ be finite, non empty subsets of $\Z$ in normal form with $|B|\ge 3$ and let $s'\ge 1$ be such that
$$|A|\ge s'(s'-1)\left(\frac{|B|}2-1\right)+s'.$$
If ${\rm diam}(B)\ge|A+B|-|A|+1$, then
$$ |A+B|\ge |A|+s'\left({|B|}-2\right)+1. $$
\end{prop}

\begin{proof}
Assume that ${\rm diam}(B)\ge|A+B|-|A|+1$.\\
Let $s$ be defined by \eqref{defs}, then the hypothesis implies $s'\le s$. According to Grynkiewicz's result
\begin{align*}
 |A+B|&\ge \left(\frac{|A|}s+\frac{|B|}2-1\right)(s+1)\ge |A|+\frac{|A|}s+\frac12\left({|B|}-2\right)(s+1)\\
&\ge |A|+s\left({|B|}-2\right)+1
\ge |A|+s'\left({|B|}-2\right)+1.
\end{align*}
\end{proof}

\begin{lem}\label{lem_stab_disc}
Let $A$ be a finite nonempty subset of $\Z$ in normal form {with at least 2 elements} such that $(A+A)\cap[0,{\rm diam}(A)]=A$.\\
Write $N={\rm diam}(A)$ and for $x\in \N$, write $A_x=A\cap[0,x]$.\\
 For $x\in A$ such that $gcd(A_x)=1$ and $k\in\N$ such that $[(k-1)x,N]\not\subset A$, we have
 $$|A_x|\le\frac{2(k-1)}k+\frac{2}{k(k+1)}\begin{cases}|A_{kx}|&\mbox{ if }kx\le N\\|A|+{kx-N}&\mbox{ if }kx>N.
 \end{cases}$$
\end{lem}
\begin{rem}
For sets $B$ with $\min(B)=0$ and $d=gcd(B)>1$, we get
$$|B\cap[0,x]|\le\frac{2(k-1)}k+\frac{2}{k(k+1)}\begin{cases}|B_{kx}|&\mbox{ if }kx\le {\rm diam}(B)\\|B|+\frac{kx-{\rm diam}(B)}{d}&\mbox{ if }kx>{\rm diam}(B)\end{cases}$$
as long as $x\in B$, $gcd(B_x)=d$ and $B\cap[(k-1)x,{\rm diam}(B)]\not=d\Z\cap [(k-1)x,{\rm diam}(B)]$.
\end{rem}
\begin{proof}
We define $c=\min\{n\in [0,N]\,:\, [n,N]\subset A\}$ and $A'=A\cup([c,+\infty)\cap\Z)$. 
We can assume $c>0$ otherwise $A=[0,N]$ and the result is trivial.\\
We get for any $k\in\N$, $kA'=A'$ and 
$$\begin{cases}
\forall y\in\N,& (kA)_y\subset A'_{y},\\
\forall y\in[0,N],& (kA)_y=A_{y}=A'_{y}.
\end{cases}$$
Let $x\in A$ such that $gcd(A_x)=1$ be fixed. We have $|A_x|\ge 3$, otherwise $gcd(A_x)\not=1$ ($1\notin A$ otherwise $A=[0,N]$ by stability and $c=0$). \\
Let $k\ge 2$ be such that $[(k-1)x,N]\not\subset A$. Then for any $\ell\in[2, k]$, we have $(\ell-1)x<c$. Let $\ell\in[2,k]$ be such that $|A_{(\ell-1)x}|\ge \frac{\ell(\ell-1)}2(|A_x|-2)+\ell$ ($\ell=2$ satisfies this inequality).\\
Since $A_{(\ell-1)x}+A_x\subset A'_{\ell x}$, by Proposition \ref{corG}, we have either
$$ |A'_{\ell x}|\geq |A_{(\ell-1)x}+A_x|\ge |A_{(\ell-1)x}|+\ell(|A_x|-2)+1$$
 or
 $$ |A'_{\ell x}|\geq |A_{(\ell-1)x}+A_x|\ge |A_{(\ell-1)x}|+x. $$
 Since $ A'_{\ell x}$ is the disjoint union of $A_{(\ell-1)x}$ and $A'\cap[(\ell-1)x+1,\ell x]$, the second inequality implies $[(\ell-1)x+1,\ell x]\subset A'$, thus $(\ell -1)x\ge c$ by stability of $A'$, contrary to the hypothesis.
 Therefore
$$ |A'_{\ell x}|\ge |A_{(\ell-1)x}|+\ell(|A_x|-2)+1\ge 
\frac{\ell(\ell+1)}2(|A_x|-2)+\ell+1.$$
We proved by iteration on $\ell\le k$ that
$$ |A'_{k x}|\ge
\frac{k(k+1)}2(|A_x|-2)+k+1.$$
 Now, $$|A'_{kx}|=\begin{cases}|A_{kx}|&\mbox{ if }kx\le N\\kx-N+|A|&\mbox{ if }kx>N\end{cases},$$
 thus we get the expected upper bound for $|A_x|$.
\end{proof}

\begin{thm}\label{stabAdisc}
Let $A$ be a finite nonempty subset of $\Z$. Write $N={\rm diam}(A)$.\\
Assume that $\min(A)=0$, $gcd(A)=1$, $(A+A)\cap[0,N]=A$ and $|A|=\lfloor\frac{N+1}{2}\rfloor+1$.\\
For $n\in [0,N]$, we write $A_n=A\cap[0,n]$. For any $n\in [2,N]$ such that $gcd(A_n)=1$, we have
$|A_n|-1\le H_n$ where $(H_n)_{n\in[2,N]}$ is the sequence defined by (see figure 4)
$$H_n=1+\left\lfloor\frac1{k+1}\left(2(n-1)-\frac{2}k\left\lfloor\frac{N}2\right\rfloor\right)\right\rfloor\quad\mbox{ for }\quad n\in\left(1+\frac{2\lfloor N/2\rfloor}k,1+\frac{2\lfloor N/2\rfloor}{k-1}\right], k\ge 2,$$
This inequality is tight: for $n, N$ fixed with $0\le n\le N$, we can find a set $A$ satisfying the hypotheses such that $|A_n|=1+ H_n$.
\end{thm}

\begin{rem}\label{exc}
The hypothesis $gcd(A_n)=1$ is necessary. Indeed if $x\in \left[2,\frac12\lfloor N/2\rfloor+1\right]$, \\
$k=\left\lceil\frac{\lfloor N/2\rfloor}{x-1}\right\rceil$ and 
$$A=\bigcup_{i=0}^{k-1}\{ix\}\cup\left[\lfloor N/2\rfloor+k, N\right]$$
then $|A|=\lfloor\frac{N+1}{2}\rfloor+1$ and $(A+A)\cap[0,N]=A$ since $kx\ge \lfloor N/2\rfloor +k$ and $(k-1)x\le \lfloor N/2\rfloor+k-1$. Therefore $A$ satisfies the hypotheses of the theorem, but $|A_{\ell x}|=\ell+1>1+H_{\ell x}$ for $\ell\leq k-1$.
\end{rem}

In the following figure, we draw (for $N=48$) in blue the piecewise affine function \\$g:x\mapsto 1+2\lfloor N/2\rfloor h\left(\frac{x-1}{2\lfloor N/2 \rfloor } \right)$ where $h$ is the function defined in Theorem \ref{thmstabA}, and in red the sequence $(H_n)_n$. The points $(n,H_n)$ correspond to the points on the grid $\Z^2$ which are just below (or on) the graph of the function $g$.
The black circles represent the points $(n,|A_n|-1)$ with $|A_n|-1>H_n$ when $A=\bigcup_{i=0}^{11}\{3i\}\cup\left[36, 48\right]$ is one of the sets described in Remark \ref{exc}.
\begin{center}
\begin{tikzpicture}[xscale=0.25,yscale=0.25]
\draw [<->] (0,27) -- (0,0) -- (50,0);
\draw [very thin] (0,0) grid (48,24);
\node [below left] at (0,0) {\small $0$};
\draw [] (0,0) -- (48,24);
\draw [dashed] (0,24) -- (48,24);
\node [left] at (0,24) {\small $N/2$};
\draw [dashed] (48,0) -- (48,24);
\node [below] at (48,0) {\small $N$};
\draw [dashed] (24,0) -- (24,48/6);
\draw [dashed] (0,48/6) -- (24,48/6);
\node [below] at (24,0) {\small $N/2$};
\node [below] at (16,0) {\small $N/3$};
\node [below] at (12,0) {\small $N/4$};
\node [below] at (8,0) {\small $N/6$};
\draw [thick, blue] (25,48/6+1) -- (49,25);
\draw[fill,red] (48,24) circle [radius=0.2] ;
\draw[fill,red] (24,48/6) circle [radius=0.2] ;
\draw[fill,red] (48/3,48/12) circle [radius=0.2] ;
\draw [thick, blue] (25,48/6+1) -- (48/3+1,48/12+1);
\draw[fill,red] (48/4,3) circle [radius=0.2] ;
\draw [thick, blue] (48/4+1,48/20+1) -- (48/3+1,48/12+1);
\draw [thick, blue] (48/4+1,48/20+1) -- (48/5+1,48/30+1);
\draw [thick, blue] (48/6+1,48/42+1) -- (48/5+1,48/30+1);
\draw [thick, blue] (48/6+1,48/42+1) -- (48/7+1,48/56+1);
\draw [thick, blue] (48/8+1,48/72+1) -- (48/7+1,48/56+1);
\draw [thick, blue] (48/8+1,48/72+1) -- (48/9+1,48/90+1);
\draw[fill,red] (48/6,1) circle [radius=0.2] ;
\draw[fill,red] (2,1) circle [radius=0.2] ;
\draw[fill,red] (3,1) circle [radius=0.2] ;
\draw[fill,red] (4,1) circle [radius=0.2] ;
\draw[fill,red] (5,1) circle [radius=0.2] ;
\draw[fill,red] (6,1) circle [radius=0.2] ;
\draw[fill,red] (7,1) circle [radius=0.2] ;

\draw[fill,red] (25,54/6) circle [radius=0.2] ;
\draw[fill,red] (26,54/6) circle [radius=0.2] ;
\draw[fill,red] (27,60/6) circle [radius=0.2] ;
\draw[fill,red] (28,66/6) circle [radius=0.2] ;
\draw[fill,red] (29,66/6) circle [radius=0.2] ;
\draw[fill,red] (30,72/6) circle [radius=0.2] ;
\draw[fill,red] (31,78/6) circle [radius=0.2] ;
\draw[fill,red] (32,78/6) circle [radius=0.2] ;
\draw[fill,red] (33,84/6) circle [radius=0.2] ;
\draw[fill,red] (34,90/6) circle [radius=0.2] ;
\draw[fill,red] (35,90/6) circle [radius=0.2] ;
\draw[fill,red] (36,96/6) circle [radius=0.2] ;
\draw[fill,red] (37,102/6) circle [radius=0.2] ;
\draw[fill,red] (38,102/6) circle [radius=0.2] ;
\draw[fill,red] (39,108/6) circle [radius=0.2] ;
\draw[fill,red] (40,114/6) circle [radius=0.2] ;
\draw[fill,red] (41,114/6) circle [radius=0.2] ;
\draw[fill,red] (42,120/6) circle [radius=0.2] ;
\draw[fill,red] (43,126/6) circle [radius=0.2] ;
\draw[fill,red] (44,126/6) circle [radius=0.2] ;
\draw[fill,red] (45,132/6) circle [radius=0.2] ;
\draw[fill,red] (46,138/6) circle [radius=0.2] ;
\draw[fill,red] (47,138/6) circle [radius=0.2] ;

\draw[fill,red] (16,4) circle [radius=0.2] ;
\draw[fill,red] (17,5) circle [radius=0.2] ;
\draw[fill,red] (18,5) circle [radius=0.2] ;
\draw[fill,red] (19,6) circle [radius=0.2] ;
\draw[fill,red] (20,6) circle [radius=0.2] ;
\draw[fill,red] (21,7) circle [radius=0.2] ;
\draw[fill,red] (22,7) circle [radius=0.2] ;
\draw[fill,red] (23,8) circle [radius=0.2] ;

\draw[fill,red] (48/3,48/12) circle [radius=0.2] ;
\draw[fill,red] (48/6,1) circle [radius=0.2] ;
\draw [dashed] (16,0) -- (16,48/12);
\draw [dashed] (0,48/12) -- (16,48/12);
\draw [dashed] (12,0) -- (12,3);
\draw [dashed] (0,3) -- (12,3);
\draw [dashed] (8,0) -- (8,1);
\draw [dashed] (0,1) -- (8,1);

\draw[fill,red] (12,3) circle [radius=0.2] ;
\draw[fill,red] (13,3) circle [radius=0.2] ;
\draw[fill,red] (14,3) circle [radius=0.2] ;
\draw[fill,red] (15,4) circle [radius=0.2] ;

\draw[fill,red] (9,2) circle [radius=0.2] ;
\draw[fill,red] (10,2) circle [radius=0.2] ;
\draw[fill,red] (11,2) circle [radius=0.2] ;
\draw[fill,red] (12,3) circle [radius=0.2] ;

\draw[thick] (3,1) circle [radius=0.3] ;
\draw[thick] (6,2) circle [radius=0.3] ;
\draw[thick] (7,2) circle [radius=0.3] ;
\draw[thick] (8,2) circle [radius=0.3] ;
\draw[thick] (9,3) circle [radius=0.3] ;
\draw[thick] (10,3) circle [radius=0.3] ;
\draw[thick] (11,3) circle [radius=0.3] ;
\draw[thick] (12,4) circle [radius=0.3] ;
\draw[thick] (13,4) circle [radius=0.3] ;
\draw[thick] (14,4) circle [radius=0.3] ;
\draw[thick] (15,5) circle [radius=0.3] ;
\draw[thick] (16,5) circle [radius=0.3] ;
\draw[thick] (18,6) circle [radius=0.3] ;
\node at (24,-3.5) {Figure 4: Graph of the sequence $(H_n)_{n\in[2,N]}$ for $N=48$.};
\end{tikzpicture}
\end{center}

\begin{proof}
Let $n\in A\cap \left(1+\frac{2\lfloor N/2\rfloor}k,1+\frac{2\lfloor N/2\rfloor}{k-1}\right]$ be such that $gcd(A_n)=1$. 
We have $kn> N$. \\
We define $c:=\min\{u\in [0,N]\,:\, [u,N]\subset A\}$ as in Lemma \ref{lem_stab_disc}.\\
Let $\ell\ge 1$ be the largest integer such that $(\ell-1) n<c$ (since $0<c\le N<kn$, $\ell$ exists and $1\le \ell\le k$). 
Since $\ell n\ge c$, we have $|A_{\ell n}|=|A|-(N-\ell n)=|A|+\ell n-N$. 
According to Lemma \ref{lem_stab_disc}, we have
 \begin{align*}
 |A_n|&\le\frac{2(\ell-1)}\ell+\frac{2}{\ell(\ell+1)}\left(|A|+\ell n-N\right)\\
 &\le 2-\frac{2}\ell+\frac{2}{\ell(\ell+1)}\left(\left\lfloor \frac{N+1}2\right\rfloor+1+\ell n-N\right)\\
& \le 2+\frac1{\ell+1}\left(2(n-1)-\frac2\ell\left(N-\left\lfloor \frac{N+1}2\right\rfloor\right)\right)\\
& \le 2+\frac1{\ell+1}\left(2(n-1)-\frac2\ell\left\lfloor \frac{N}2\right\rfloor\right)\\
& \le 2+u_\ell (n-1,2\lfloor N/2\rfloor)
 \end{align*}
 where $(u_i)_i$ is the sequence defined in \eqref{def_ui}. \\
 Using $u_k(x,d)=\max_{i}u_i(x,d)$ if 
 $(k-1)x\le d<kx$, we get 
 $$|A_n|\le 2+\frac1{k+1}\left(2(n-1)-\frac2{k}\left\lfloor \frac{N}2\right\rfloor\right)=2+2\lfloor N/2\rfloor h\left(\frac{n-1}{2\lfloor N/2\rfloor}\right),$$
 where $h$ is the function defined in Theorem \ref{thmstabA},
 and $|A_n|$ being integer, we finally get
 $|A_n|-1\le H_n$.

If $n\in [1,N]\cap A^c$ and $n'=\max(A_n)$ (thus $n'\le n-1$), we have $A_n=A_{n'}$. The function $h$ being increasing, 
$(H_n)_{n\in[1,N]}$ is nondecreasing and 
$$|A_n|-1 =|A_{n'}|-1\le H_{n'}\le H_n.$$
The inequality is proved in any case.
\bigskip

This upper bound is sharp as we prove now.\\
Let $n\in [2,N]$ be fixed and $k\in \N$ be such that $n\in \left(1+\frac2k\lfloor N/2\rfloor,1+\frac2{k-1}\lfloor N/2\rfloor\right]$. Assume that $H_n\ge 2$ (otherwise $gcd(A_n)\not=1$).\\
By definition of $H_n$,
$H_n-1\le\frac1{k+1}\left(2(n-1)-\frac2k\lfloor N/2\rfloor\right)=u_k$.\\
Assume first that $k\ge 3$, i.e. $n\le 1+\lfloor N/2\rfloor $, then $u_1=n-1-\lfloor N/2\rfloor<H_n-1$.\\ 
Since $(u_i)_{1\le i\le k}$ is nondecreasing, we can fix $\ell\in[2,k]$ minimal such that $u_{\ell-1}< H_n-1\le u_\ell$.
We define the set
\begin{equation}\label{an}
A=\bigcup_{i=0}^{\ell-1}\left[i\left(n-H_n+1\right),in\right]\bigcup\left[\frac{\ell(\ell-1)}2(H_n-1)+\ell+\left\lfloor \frac{N}2\right\rfloor ,N\right].
\end{equation}
We have $|A_n|=1+H_n$.\\
Since $H_n-1\le u_\ell$, we have $\ell(n-H_n+1)\ge\frac{\ell(\ell-1)}2(H_n-1)+\ell+\left\lfloor \frac{N}2\right\rfloor$ thus\\ 
$(A+A)\cap[0,N]=A$.\\
Since $H_n-1>u_{\ell-1}$, we have $(\ell-1) n<\frac{\ell(\ell-1)}2(H_n-1)+\ell+\left\lfloor \frac{N}2\right\rfloor$.\\ Furthermore, 
$n-1\le\frac2{k-1}\lfloor N/2\rfloor$, thus $2\lfloor N/2\rfloor\ge(n-1)(k-1)\ge(n-1)(\ell-1)$ and 
$$H_n-1\le u_\ell=\frac{1}{\ell+1}\left(2(n-1)-\frac2\ell\lfloor N/2\rfloor\right)\le \frac{n-1}{\ell+1}\left(2-\frac{\ell-1}{\ell}\right)=\frac{n-1}\ell\le\frac{n-1}{\ell-1},$$ 
which yields $(\ell-1)(n-H_n+1)>(\ell-2)n$ and 
\begin{align*}
|A|&=\sum_{i=0}^{\ell-1}\left(i(H_n-1)+1\right)+N-\frac{\ell(\ell-1)}2(H_n-1)-\ell-\left\lfloor \frac{N}2\right\rfloor+1 \\
&=N-\left\lfloor \frac{N}2\right\rfloor+1 =1+\left\lfloor \frac{N+1}2\right\rfloor.
\end{align*}
This set verifies all conditions of Theorem \ref{stabAdisc} and is such that $|A_n|=H_n+1$.\\
If $k=2$, then $H_n-1=u_1+\left\lfloor\frac13(2\lfloor N/2\rfloor-(n-1))\right\rfloor\ge u_1=n-1-\lfloor N/2\rfloor$.\\
In this case, we define
$$A=\{0\}\cup \left[\left(n-H_n+1\right),n\right]\cup\left[(H_n-1)+2+\left\lfloor \frac{N}2\right\rfloor ,N\right],$$
which also verifies all conditions of Theorem \ref{stabAdisc} and is such that $|A_n|=H_n+1$.\\
\end{proof}

\begin{rem}
In this discrete case, for a fixed diameter $N$ and a fixed $n\in [2,N]$ we may have several left-stable sets in normal form of diameter $N$ satisfying $|A_n|=H_n+1$ and there may exists some set $A$ such that $|A_n|=H_n$ and $|A_m|=H_m$ with $m\not=n$.\\
For example, if $H_n=H_{n-1}$,  and $A(n)$ is the set defined by \eqref{an}, $A(n-1)$ the set defined by \eqref{an} replacing $n$ by $n-1$ and $k$ by $k'$ such that $n-1\in \left(1+\frac2{k'}\lfloor N/2\rfloor,1+\frac2{k'-1}\lfloor N/2\rfloor\right]$, we have
$$|A(n)_{n}|=H_{n}=H_{n-1}=|A(n-1)_{n-1}|=|A(n-1)_{n}| \, \mbox{ since }n\not\in A(n-1).$$
\\
We illustrate these phenomenons with the following example.\\
If $N=48$ and $n=13$, then $H_{13}=H_{12}=3$ and the sets
$$A=A(13)=\{0\}\cup[11,13]\cup [22,26] \cup [33,48]$$
and
$$B=A(12)=\{0\}\cup[10,12]\cup [20,24] \cup [30,36]\cup [40,48]$$
both satisfy all conditions of Theorem \ref{stabAdisc} and are such that $|A_{13}|=|B_{13}|=H_{13}+1$.\\

\end{rem}

In the following figure, we represent (for $N=48$) with black points the sequence $(H_n)_{2\le n\le 48}$, with blue points the sequence $(|A_n|-1)_{1\le n\le 48}$ and under the graph the set $A$, with purple circles the sequence $(|B_n|-1)_{1\le n\le 48}$ and under the graph the set $B$.
\begin{center}
\begin{tikzpicture}[xscale=0.25,yscale=0.25]
\draw [<->] (0,27) -- (0,0) -- (50,0);
\node [below left] at (0,0) {\small $0$};
\draw [] (0,0) -- (48,24);
\draw [dashed] (0,24) -- (48,24);
\node [left] at (0,24) {\small $N/2=24$};
\draw [dashed] (48,0) -- (48,24);
\node [below] at (48,0) {\small $N=48$};
\draw [dashed] (13,0) -- (13,3);
\node [below] at (13,0) {\small $n=13$};

\draw[fill]  (48,24) circle [radius=0.1] ;
\draw[fill]  (24,48/6) circle [radius=0.1] ;
\draw[fill]  (48/3,48/12) circle [radius=0.1] ;
\draw[fill]  (48/4,3) circle [radius=0.1] ;
\draw[fill]  (48/6,1) circle [radius=0.1] ;
\draw[fill]  (2,1) circle [radius=0.1] ;
\draw[fill]  (3,1) circle [radius=0.1] ;
\draw[fill]  (4,1) circle [radius=0.1] ;
\draw[fill]  (5,1) circle [radius=0.1] ;
\draw[fill]  (6,1) circle [radius=0.1] ;
\draw[fill]  (7,1) circle [radius=0.1] ;
\node[above] at (26,10) {\small $H_n$} ;

\draw[fill]  (25,54/6) circle [radius=0.1] ;
\draw[fill]  (26,54/6) circle [radius=0.1] ;
\draw[fill]  (27,60/6) circle [radius=0.1] ;
\draw[fill]  (28,66/6) circle [radius=0.1] ;
\draw[fill]  (29,66/6) circle [radius=0.1] ;
\draw[fill]  (30,72/6) circle [radius=0.1] ;
\draw[fill]  (31,78/6) circle [radius=0.1] ;
\draw[fill]  (32,78/6) circle [radius=0.1] ;
\draw[fill]  (33,84/6) circle [radius=0.1] ;
\draw[fill]  (34,90/6) circle [radius=0.1] ;
\draw[fill]  (35,90/6) circle [radius=0.1] ;
\draw[fill]  (36,96/6) circle [radius=0.1] ;
\draw[fill]  (37,102/6) circle [radius=0.1] ;
\draw[fill]  (38,102/6) circle [radius=0.1] ;
\draw[fill]  (39,108/6) circle [radius=0.1] ;
\draw[fill]  (40,114/6) circle [radius=0.1] ;
\draw[fill]  (41,114/6) circle [radius=0.1] ;
\draw[fill]  (42,120/6) circle [radius=0.1] ;
\draw[fill]  (43,126/6) circle [radius=0.1] ;
\draw[fill]  (44,126/6) circle [radius=0.1] ;
\draw[fill]  (45,132/6) circle [radius=0.1] ;
\draw[fill]  (46,138/6) circle [radius=0.1] ;
\draw[fill]  (47,138/6) circle [radius=0.1] ;

\draw[fill]  (16,4) circle [radius=0.1] ;
\draw[fill]  (17,5) circle [radius=0.1] ;
\draw[fill]  (18,5) circle [radius=0.1] ;
\draw[fill]  (19,6) circle [radius=0.1] ;
\draw[fill]  (20,6) circle [radius=0.1] ;
\draw[fill]  (21,7) circle [radius=0.1] ;
\draw[fill]  (22,7) circle [radius=0.1] ;
\draw[fill]  (23,8) circle [radius=0.1] ;

\draw[fill]  (48/3,48/12) circle [radius=0.1] ;
\draw[fill]  (48/6,1) circle [radius=0.1] ;

\draw[fill]  (12,3) circle [radius=0.1] ;
\draw[fill]  (13,3) circle [radius=0.1] ;
\draw[fill]  (14,3) circle [radius=0.1] ;
\draw[fill]  (15,4) circle [radius=0.1] ;

\draw[fill]  (9,2) circle [radius=0.1] ;
\draw[fill]  (10,2) circle [radius=0.1] ;
\draw[fill]  (11,2) circle [radius=0.1] ;
\draw[fill]  (12,3) circle [radius=0.1] ;

\draw[fill, purple] (0,-5) circle [radius=0.1] ;
\draw[fill, purple]  (10,-5) circle [radius=0.1] ;
\draw[fill, purple]  (11,-5) circle [radius=0.1] ;
\draw[fill, purple]  (12,-5) circle [radius=0.1] ;
\draw[fill, purple]  (20,-5) circle [radius=0.1] ;
\draw[fill, purple]  (21,-5) circle [radius=0.1] ;
\draw[fill, purple]  (22,-5) circle [radius=0.1] ;
\draw[fill, purple]  (23,-5) circle [radius=0.1] ;
\draw[fill, purple]  (24,-5) circle [radius=0.1] ;
\draw[fill, purple]  (30,-5) circle [radius=0.1] ;
\draw[fill, purple]  (31,-5) circle [radius=0.1] ;
\draw[fill, purple]  (32,-5) circle [radius=0.1] ;
\draw[fill, purple]  (33,-5) circle [radius=0.1] ;
\draw[fill, purple]  (34,-5) circle [radius=0.1] ;
\draw[fill, purple]  (35,-5) circle [radius=0.1] ;
\draw[fill, purple]  (36,-5) circle [radius=0.1] ;
\draw[fill, purple]  (40,-5) circle [radius=0.1] ;
\draw[fill, purple]  (41,-5) circle [radius=0.1] ;
\draw[fill, purple]  (42,-5) circle [radius=0.1] ;
\draw[fill, purple]  (43,-5) circle [radius=0.1] ;
\draw[fill, purple]  (44,-5) circle [radius=0.1] ;
\draw[fill, purple]  (45,-5) circle [radius=0.1] ;
\draw[fill, purple]  (46,-5) circle [radius=0.1] ;
\draw[fill, purple]  (47,-5) circle [radius=0.1] ;
\draw[fill, purple]  (48,-5) circle [radius=0.1] ;

\draw[thick, purple] (1,0) circle [radius=0.25] ;
\draw[thick, purple] (2,0) circle [radius=0.25] ;
\draw[thick, purple] (3,0) circle [radius=0.25] ;
\draw[thick, purple] (4,0) circle [radius=0.25] ;
\draw[thick, purple] (5,0) circle [radius=0.25] ;
\draw[thick, purple] (6,0) circle [radius=0.25] ;
\draw[thick, purple] (7,0) circle [radius=0.25] ;
\draw[thick, purple] (8,0) circle [radius=0.25] ;
\draw[thick,purple] (9,0) circle [radius=0.25] ;
\draw[thick,purple] (10,1) circle [radius=0.25] ;
\draw[thick,purple] (11,2) circle [radius=0.25] ;
\draw[thick,purple] (12,3) circle [radius=0.25] ;
\draw[thick, purple] (13,3) circle [radius=0.25] ;
\draw[thick, purple] (14,3) circle [radius=0.25] ;
\draw[thick, purple] (15,3) circle [radius=0.25] ;
\draw[thick, purple] (16,3) circle [radius=0.25] ;
\draw[thick, purple] (17,3) circle [radius=0.25] ;
\draw[thick, purple] (18,3) circle [radius=0.25] ;
\draw[thick, purple] (19,3) circle [radius=0.25] ;
\draw[thick, purple] (20,4) circle [radius=0.25] ;
\draw[thick, purple] (21,5) circle [radius=0.25] ;
\draw[thick, purple] (22,6) circle [radius=0.25] ;
\draw[thick, purple] (23,7) circle [radius=0.25] ;
\draw[thick, purple] (24,8) circle [radius=0.25] ;
\draw[thick, purple] (25,8) circle [radius=0.25] ;
\draw[thick, purple] (26,8) circle [radius=0.25] ;
\draw[thick, purple] (27,8) circle [radius=0.25] ;
\draw[thick, purple] (28,8) circle [radius=0.25] ;
\draw[thick, purple] (29,8) circle [radius=0.25] ;
\draw[thick, purple] (30,9) circle [radius=0.25] ;
\node [right, purple] at (25,6) {\small $|B_n|-1$};
\draw[thick, purple]  (31,10) circle [radius=0.25]  ;
\draw[thick, purple] (32,11) circle [radius=0.25]  ;
\draw[thick, purple] (33,12) circle [radius=0.25]  ;
\draw[thick, purple] (34,13) circle [radius=0.25]  ;
\draw[thick, purple] (35,14) circle [radius=0.25]  ;
\draw[thick, purple]  (36,15) circle [radius=0.25]  ;
\draw[thick, purple]  (37,15) circle [radius=0.25]  ;
\draw[thick, purple]  (38,15) circle [radius=0.25]  ;
\draw[thick, purple]  (39,15) circle [radius=0.25]  ;
\draw[thick, purple]  (40,16) circle [radius=0.25]  ;
\draw[thick, purple]  (41,17) circle [radius=0.25]  ;
\draw[thick, purple]  (42,18) circle [radius=0.25]  ;
\draw[thick, purple]  (43,19) circle [radius=0.25]  ;
\draw[thick, purple]  (44,20) circle [radius=0.25]  ;
\draw[thick, purple]  (45,21) circle [radius=0.25]  ;
\draw[thick, purple]  (46,22) circle [radius=0.25]  ;
\draw[thick, purple]  (47,23) circle [radius=0.25]  ;
\draw[thick, purple]  (48,24) circle [radius=0.25]  ;

\draw[thick, purple]  (11,-5) circle [radius=0.1] ;
\draw[thick, purple]  (12,-5) circle [radius=0.1] ;
\draw[thick, purple]  (20,-5) circle [radius=0.1] ;
\draw[thick, purple]  (21,-5) circle [radius=0.1] ;
\draw[thick, purple]  (22,-5) circle [radius=0.1] ;
\draw[thick, purple]  (23,-5) circle [radius=0.1] ;
\draw[thick, purple]  (24,-5) circle [radius=0.1] ;
\draw[thick, purple]  (30,-5) circle [radius=0.1] ;
\draw[thick, purple]  (31,-5) circle [radius=0.1] ;
\draw[thick, purple]  (32,-5) circle [radius=0.1] ;
\draw[thick, purple]  (33,-5) circle [radius=0.1] ;
\draw[thick, purple]  (34,-5) circle [radius=0.1] ;
\draw[thick, purple]  (35,-5) circle [radius=0.1] ;
\draw[thick, purple]  (36,-5) circle [radius=0.1] ;
\draw[thick, purple]  (40,-5) circle [radius=0.1] ;
\draw[thick, purple]  (41,-5) circle [radius=0.1] ;
\draw[thick, purple]  (42,-5) circle [radius=0.1] ;
\draw[thick, purple]  (43,-5) circle [radius=0.1] ;
\draw[thick, purple]  (44,-5) circle [radius=0.1] ;
\draw[thick, purple]  (45,-5) circle [radius=0.1] ;
\draw[thick, purple]  (46,-5) circle [radius=0.1] ;
\draw[thick, purple]  (47,-5) circle [radius=0.1] ;
\draw[thick, purple]  (48,-5) circle [radius=0.1] ;

\node [right, purple] at (49,-5) {$B$};
\node [right, blue] at (49,-3) {$A$};

\draw[fill,blue] (0,-3) circle [radius=0.1] ;
\draw[fill,blue] (11,-3) circle [radius=0.1] ;
\draw[fill,blue] (12,-3) circle [radius=0.1] ;
\draw[fill,blue] (13,-3) circle [radius=0.1] ;
\draw[fill,blue] (22,-3) circle [radius=0.1] ;
\draw[fill,blue] (23,-3) circle [radius=0.1] ;
\draw[fill,blue] (24,-3) circle [radius=0.1] ;
\draw[fill,blue] (25,-3) circle [radius=0.1] ;
\draw[fill,blue] (26,-3) circle [radius=0.1] ;
\draw[fill,blue] (33,-3) circle [radius=0.1] ;
\draw[fill,blue] (34,-3) circle [radius=0.1] ;
\draw[fill,blue] (35,-3) circle [radius=0.1] ;
\draw[fill,blue] (36,-3) circle [radius=0.1] ;
\draw[fill,blue] (37,-3) circle [radius=0.1] ;
\draw[fill,blue] (38,-3) circle [radius=0.1] ;
\draw[fill,blue] (39,-3) circle [radius=0.1] ;
\draw[fill,blue] (40,-3) circle [radius=0.1] ;
\draw[fill,blue] (41,-3) circle [radius=0.1] ;
\draw[fill,blue] (42,-3) circle [radius=0.1] ;
\draw[fill,blue] (43,-3) circle [radius=0.1] ;
\draw[fill,blue] (44,-3) circle [radius=0.1] ;
\draw[fill,blue] (45,-3) circle [radius=0.1] ;
\draw[fill,blue] (46,-3) circle [radius=0.1] ;
\draw[fill,blue] (47,-3) circle [radius=0.1] ;
\draw[fill,blue] (48,-3) circle [radius=0.1] ;
\node[blue] at (37,9) {\small $|A_n|-1$};
\draw[fill,blue] (1,0) circle [radius=0.1] ;
\draw[fill,blue] (2,0) circle [radius=0.1] ;
\draw[fill,blue] (3,0) circle [radius=0.1] ;
\draw[fill,blue] (4,0) circle [radius=0.1] ;
\draw[fill,blue] (5,0) circle [radius=0.1] ;
\draw[fill,blue] (6,0) circle [radius=0.1] ;
\draw[fill,blue] (7,0) circle [radius=0.1] ;
\draw[fill,blue] (8,0) circle [radius=0.1] ;
\draw[fill,blue] (9,0) circle [radius=0.1] ;
\draw[fill,blue] (10,0) circle [radius=0.1] ;
\draw[fill,blue] (11,1) circle [radius=0.1] ;
\draw[fill,blue] (12,2) circle [radius=0.1] ;
\draw[fill,blue] (13,3) circle [radius=0.1] ;
\draw[fill,blue] (14,3) circle [radius=0.1] ;
\draw[fill,blue] (15,3) circle [radius=0.1] ;
\draw[fill,blue] (16,3) circle [radius=0.1] ;
\draw[fill,blue] (17,3) circle [radius=0.1] ;
\draw[fill,blue] (18,3) circle [radius=0.1] ;
\draw[fill,blue] (19,3) circle [radius=0.1] ;
\draw[fill,blue] (20,3) circle [radius=0.1] ;
\draw[fill,blue] (21,3) circle [radius=0.1] ;
\draw[fill,blue] (22,4) circle [radius=0.1] ;
\draw[fill,blue] (23,5) circle [radius=0.1] ;
\draw[fill,blue] (24,6) circle [radius=0.1] ;
\draw[fill,blue] (25,7) circle [radius=0.1] ;
\draw[fill,blue] (26,8) circle [radius=0.1] ;
\draw[fill,blue] (27,8) circle [radius=0.1] ;
\draw[fill,blue] (28,8) circle [radius=0.1] ;
\draw[fill,blue] (29,8) circle [radius=0.1] ;
\draw[fill,blue] (30,8) circle [radius=0.1] ;
\draw[fill,blue] (31,8) circle [radius=0.1] ;
\draw[fill,blue] (32,8) circle [radius=0.1] ;
\draw[fill,blue] (33,9) circle [radius=0.1] ;
\draw[fill,blue] (34,10) circle [radius=0.1] ;
\draw[fill,blue] (35,11) circle [radius=0.1] ;
\draw[fill,blue] (36,12) circle [radius=0.1] ;
\draw[fill,blue] (37,13) circle [radius=0.1] ;
\draw[fill,blue] (38,14) circle [radius=0.1] ;
\draw[fill,blue] (39,15) circle [radius=0.1] ;
\draw[fill,blue] (40,16) circle [radius=0.1] ;
\draw[fill,blue] (41,17) circle [radius=0.1] ;
\draw[fill,blue] (42,18) circle [radius=0.1] ;
\draw[fill,blue] (43,19) circle [radius=0.1] ;
\draw[fill,blue] (44,20) circle [radius=0.1] ;
\draw[fill,blue] (45,21) circle [radius=0.1] ;
\draw[fill,blue] (46,22) circle [radius=0.1] ;
\draw[fill,blue] (47,23) circle [radius=0.1] ;
\draw[fill,blue] (48,24) circle [radius=0.1] ;

\node at (24,-10) {Figure 5: An example of two sets $A$ and $B$ satisfying $1+H_n=|A_n|=|B_n|$.};
\end{tikzpicture}
\end{center}

\bibliographystyle{alpha}
\bibliography{stable_sets_nov17}

\end{document}